\documentclass[12pt]{amsart}

\usepackage{amsmath, amsthm, amssymb}
\usepackage{graphicx}
\usepackage{mathbbol}
\usepackage{txfonts}
\usepackage[usenames]{color}
\usepackage{pgfplots}
\usepackage{comment}
\usepackage[mathscr]{eucal}
\usepackage{mathrsfs}
\usepackage{amsfonts}

\allowdisplaybreaks[4]

\newtheorem{thm}{Theorem}[section]
\newcommand{\bt}{\begin{thm}}
\newcommand{\et}{\end{thm}}

\newtheorem{conj}[thm]{Conjecture}

\newtheorem{cor}[thm]{Corollary}   
\newcommand{\bc}{\begin{cor}}
\newcommand{\ec}{\end{cor}}

\newtheorem{lem}[thm]{Lemma}   
\newcommand{\bl}{\begin{lem}}
\newcommand{\el}{\end{lem}}

\newtheorem{prop}[thm]{Proposition}
\newcommand{\bp}{\begin{prop}}
\newcommand{\ep}{\end{prop}}

\newtheorem{defn}[thm]{Definition}
\newcommand{\bd}{\begin{defn}}    
\newcommand{\ed}{\end{defn}}

\newtheorem{rmrk}[thm]{Remark}   
\newcommand{\br}{\begin{rmrk}}
\newcommand{\er}{\end{rmrk}}

\newtheorem{example}[thm]{Example}

\newcommand{\Fto}{\stackrel {\mathcal{F}}{\longrightarrow} }

\newcommand{\mina}{\operatorname{MinA}}

\newcommand{\Scal}{\operatorname{Scalar}}

\newcommand{\be}{\begin{equation}}

 \newcommand{\ee}{\end{equation}}

\newcommand{\R}{\mathbb{R}}

\newcommand{\diam}{\operatorname{Diam}}

\newcommand{\mass}{{\mathbf M}}

\newcommand{\area}{\operatorname{Area}}
\newcommand{\vol}{\operatorname{Vol}}

\newcommand{\ds}{\displaystyle}

\begin{document}

\title[]{A Compactness Theorem for \\
Rotationally Symmetric Riemannian \\Manifolds with
Positive Scalar Curvature}


\author[]{Jiewon Park}
\address{Jiewon Park, 182 Memorial Drive, Cambridge, Massachusetts 02139}
\email{jiewon@mit.edu}

\author[]{Wenchuan Tian}
\address{Wenchuan Tian, 619 Red Cedar Road, C212 Wells Hall, East Lansing, Michigan 48824}
\email{tian.wenchuan@gmail.com}

\author[]{Changliang Wang}
\address{Changliang Wang, Max-Planck-Institut f\"ur Mathematik, Vivatsgasse 7, Bonn 53111, Germany}
\email{cwangmath@outlook.com}


\date{revised \today}

\begin{abstract}
Gromov and Sormani conjectured that sequences of compact Riemannian manifolds with
nonnegative scalar curvature and area of minimal surfaces bounded below should have subsequences which converge in the intrinsic flat sense to limit spaces which have nonnegative generalized scalar curvature and Euclidean tangent cones almost everywhere. In this paper we prove this conjecture for sequences of rotationally symmetric warped product manifolds. We show that the limit spaces have $H^1$ warping function that has nonnegative scalar curvature in a weak sense, and have Euclidean tangent cones almost everywhere.
\end{abstract}

\maketitle

\section{\bf Introduction}

In \cite{Gro14a} and \cite{Gro14b}, Gromov conjectured that the intrinsic flat convergence
may preserve a generalized notion of nonnegative scalar curvature. In light of this and examples
constructed by Basilio, Dodziuk, and Sormani in \cite{BDS17}, Gromov and Sormani proposed the following conjecture in \cite{GS18} (see also \cite{Sormani-scalar}).

\begin{conj}\label{Scalar-Compactness}
Let $\{M_j^3\}$ be a sequence of
 closed oriented manifolds without boundary satisfying
\be
 \vol(M_j) \le V,
\ee
\be
 \diam(M_j) \le D,
\ee
\be
\Scal_j \ge 0,
\ee
and
\be
 \mina(M_j) \ge A>0.
\ee
Here, $\mina(M_j)$ is defined as the infimum of areas of closed embedded minimal surfaces on $M_j$. Then a subsequence of $\{M_j\}$ converges in the volume preserving intrinsic flat sense to a limit space $M_\infty$,
\be
M_{j_k} \Fto M_\infty \ \ \textrm{ and } \ \ \mass(M_{j_k}) \to \mass(M_\infty).
\ee
Moreover, $M_\infty$ has nonnegative generalized scalar curvature, has Euclidean tangent cones almost everywhere, and satisfies the prism inequality.
\end{conj}

The convergence in Conjecture $\ref{Scalar-Compactness}$ is under the Sormani-Wenger intrinsic flat (SWIF) distance between integral current spaces introduced by Sormani and Wenger in \cite{SW11}. In this paper, we will prove the first two parts of Conjecture $\ref{Scalar-Compactness}$ in the case when $M_{j}$ are rotationally symmetric Riemannian manifolds. Namely, we prove the convergence to a smooth manifold with $\mathcal{C}^0$ metric which has Euclidean tangent cones almost everywhere and nonnegative scalar curvature in the sense of distributions.

We briefly recall the notion of the intrinsic flat distance following \cite{Sormani-scalar}. An integral current space $(X, d, T )$ is a metric space $(X,d)$ with an integral current structure $T$. An oriented Riemannian manifold $(M^{m}, g)$ of finite volume can be naturally viewed as an integral current space, since it has a natural metric induced by the Riemannian metric $g$, and an integral current structure $T$ acting on differential $m$-forms $\omega$ as
\be
T(\omega)=\int_{M}\omega.
\ee

The mass of an integral current space $\mass(T)$ can be understood as a weighted volume. When the integral current space is an oriented Riemannian manifold, its mass is just its volume, $\mass(M) ={\rm Vol}(M)$. The boundary of an integral current space was defined by Ambrosio and Kirchheim in \cite{AK00} so that it satisfies Stokes' Theorem. In particular, when the integral current space is a Riemannian manifold $M$, then its boundary is just the usual boundary $\partial M$. We refer to \cite{AK00} for more details about integral current spaces.

Let $Z$ be a metric space and $T_{1}$ and $T_{2}$ be two $m$-integral currents on
$Z$.  Recall the flat distance between integral currents $T_{1}$ and $T_{2}$ defined by Federer and Fleming in \cite{FF60} is
\be
d^{Z}_{F}(T_{1}, T_{2})=\inf\bigg\{\mass(B^{m+1})+\mass(A^{m}) \mid T_{1}-T_{2}=A+\partial B\bigg\}.
\ee

\bd[\cite{SW11}]
The SWIF distance between integral current spaces $(X_{1}, d_{1}, T_{1})$, and $(X_{2}, d_{2}, T_{2})$ is defined as
\be
d_{\mathcal{F}}((X_{1}, d_{1}, T_{1}), (X_{2}, d_{2}, T_{2})) = \inf \bigg\{ d^{Z}_{F}(\varphi_{1\#}T_1, \varphi_{2\#}T_2) \mid \varphi_{i} : X_{i}\rightarrow Z\bigg\},
\ee
where the infimum is taken over all common complete metric spaces $Z$ and
all isometric embeddings $\varphi_{i} : X_{i}\rightarrow Z$, where $\varphi_{i\#}$ is the push-forward
map on integral currents.
\ed

We refer to \cite{SW11} for properties of the SWIF distances and only mention Wenger's Compactness Theorem \cite{Wen11}, which says that if a sequence of Riemannian manifolds $M_{j}$ satisfies
\be
\diam(M_j) \le D,
\ee
\be
\vol(M_j) \le V,
\ee
\be
\area(\partial M_j) \le A,
\ee
then there exists a subsequence $M_{j_i}$ such that $M_{j_i} \Fto M_\infty$, where $M_\infty$ is an integral current space, possibly the $0$ space. In Ambrosio and Kirchheim's work \cite{AK00} and Sormani and Wenger's work \cite{SW11}, it has been shown that $M_\infty$ can have tangent cones that are normed spaces. Note that there is no hypothesis on $\mina$ or scalar curvature in this compactness theorem. In \cite{AS18}, Allen and Sormani constructed rotationally symmetric examples with $\mina$ bound, but without scalar curvature bound, which have non-Euclidean tangent cones.

If $M_j$ has nonnegative scalar curvature, Gromov  proved in \cite{Gro14b} that if the limit space is smooth and the convergence is $\mathcal{C}^0$, then indeed the scalar curvature is nonnegative on the limit space.  In \cite{Bam16}, Bamler proved the same result using Ricci flow. In general, the volume is only lower semicontinuous; collapsing, or cancellation, can happen even with a scalar curvature bound, and the mass of the limit space can be 0. Such examples are given in Example \ref{collapsing} and by Sormani and Wenger in \cite{SW10}. Also note that the SWIF limit does not always coincide with the Gromov-Hausdorff limit; see Example \ref{lakzian}, which is an example constructed by Lakzian.

Now we consider rotationally symmetric Riemannian manifolds $(M^3_j, g_j)$, that is, $M^{3}_j$ are diffeomorphic to $\mathbb{S}^3$ with the metric tensor
\be
g_j = ds^2 + f_j(s)^2 g_{{\mathbb{S}}^2},
\ee
where $0\leq s\leq D_{j}$, and $f_j(s)$ is a smooth nonnegative function with $f_j(0) = f_j(D_{j})=0$ and $f_j>0$ everywhere else, and $f_j'(0) = 1$, $f_j'(D_{j}) = -1$, so that the metric tensor is smooth. Our main result confirms Conjecture $\ref{Scalar-Compactness}$ in this rotationally symmetric setting.
\bt\label{main-result}
Let $(M_j^3, g_{j})$ be a sequence of oriented rotationally symmetric Riemannian manifolds without boundary satisfying
\be
\diam(M_j) \le D,
\ee
\be
\Scal_j \ge 0,
\ee
\be
\mina(M_j) \ge A>0,
\ee
then a subsequence converges in the volume preserving intrinsic flat sense to a metric space $(M_\infty,g_\infty)$
\be
M_{j_k} \Fto M_\infty \ \ \textrm{ and } \ \ \mass(M_{j_k}) \to \mass(M_\infty).
\ee
The metric $g_{\infty}$ is rotationally symmetric, $\mathcal{C}^0$, $H^1$, and has nonnegative generalized scalar curvature, meaning that the warping function satisfies the inequality in Lemma $\ref{lem-scal}$ in the sense of distributions.
\et

\br\label{Scal-ge-6}
{\rm
In Theorem $\ref{main-result}$, when $\Scal_j\ge 0$ is replaced by $\Scal_j\ge H > 0$, we have the same convergence result and that $M_\infty$ has generalized scalar curvature at least $H$ in the sense of distributions.
}
\er

\br
{\rm
Note that in Theorem $\ref{main-result}$ we do not need to assume a uniform upper bound on volume as in Conjecture $\ref{Scalar-Compactness}$. Actually, with the help of Lemma $\ref{newlem}$, a uniform upper bound of volume follows from the non-negativity of scalar curvature and uniform upper bound of diameter. Lemma $\ref{newlem}$ also implies that the tangent cones are Euclidean almost everywhere on the limit space.
}
\er

\br
{\rm
In Section 3, we will illustrate that if $\mina(M_{j})$ has no positive lower bound then the sequence $M_{j}$ could collapse to the zero current. We will also recall an interesting example obtained by Lakzian in \cite{Lak16} to illustrate the difference between SWIF limit and Gromov-Hausdorff limit of sequences of rotationally symmetric Riemannian manifolds satisfying hypotheses in Theorem~\ref{main-result}.
}
\er

The SWIF convergence has been applied to study sequences of warped product type Riemannian manifolds with non-negative scalar curvature in various interesting problems, see Lee-Sormani \cite{LeeSormani1}, LeFloch-Sormani \cite{LeFloch-Sormani-1}, and Allen-Hernandez-Vazquez-Parise-Payne-Wang \cite{AHPPW18}. Especially, LeFloch and Sormani \cite{LeFloch-Sormani-1} proved a compactness theorem for Hawking mass in the rotationally symmetric setting. They proved that for a sequence of  three dimensional oriented Riemannian manifolds $M_j^3$ with boundary, with nonnegative scalar curvature and certain bounds including a bound on Hawking mass, a subsequence converges in the volume preserving SWIF distance to a limit space with nonnegative generalized scalar curvature in a generalized sense. This theorem is proved by showing $H^1_{loc}$
convergence of a subsequence of the manifolds with a well chosen
gauge and showing that the $H^1_{loc}$ limit coincides with a $\mathcal{F}$ limit
using Theorem~\ref{thm-subdiffeo}.
In general it is unknown whether $H^1_{loc}$ convergence implies
$\mathcal{F}$ convergence, but the monotonicity of the
Hawking mass allows for the implication in this setting. The limit space is a rotationally symmetric
manifold with a metric tensor $g \in H^1_{loc}$ and it is possible to define
generalized notions of nonnegative scalar curvature in a weak sense.

The organization of this paper is as follows. In section 2 we derive some basic consequences from the hypotheses on volume, diameter, scalar curvature, and $\mina$. In section 3 some interesting examples on the SWIF convergence are given to better illustrate the notion. In section 4 we prove uniform convergence of $f_j$ to a limit function $f_\infty$ and construct the limit space $M_\infty$ (Theorem \ref{limfun}). Then we show that $M_\infty$ has Euclidean tangent cones almost everywhere (Theorem \ref{tangcone}), and nonnegative generalized scalar curvature (Theorem \ref{gen-scalar-mid}). Here, we use the notion of distributional scalar cuvature, which is studied by LeFloch and Mardare in \cite{LM07}. Finally, in section 5 we prove the SWIF convergence of $M_j$ to $M_\infty$ after taking a subsequence (Theorem \ref{swif}). The proof relies on the technique of identifying large diffeomorphic regions on $M_j$ and $M_\infty$, introduced by Lakzian and Sormani \cite{Lakzian-Sormani}.

It remains an interesting open question to prove or disprove the prism inequality on the limit space. This question is so challenging even for smooth metric spaces that it was only recently settled by Li in \cite{Li17}.

{\bf Acknowledgements.}
This project began at the Summer School for Geometric Analysis as a part of Thematic Program on Geometric Analysis in July 2017 at Fields Institute for Research in Mathematical Science. The authors would like to thank Spyros Alexakis, Walter Craig, Robert Haslhofer, Spiro Karigiannis, and McKenzie Wang for organizing the program. The authors also thank Fields Institute for providing an excellent research environment. The authors would like to express their gratitude to Christina Sormani for giving an inspiring course on the subject in the summer school, grouping us together for the project, and her constant support and guidance through numerous helpful discussions. The authors are grateful to Brian Allen, Christian Ketterer, Chen-Yun Lin, and Raquel Perales for serving as teaching assistants during Sormani's course. The authors would like to thank Hanci Chi for many useful discussions. The authors participated in workshops to work on this project, which were funded by Sormani's NSF grant DMS-1309360 and DMS-1612049.

\section{\bf Basic Consequences of the Hypotheses}

In this section, we derive basic consequences from the hypotheses in Theorem $\ref{main-result}$.
Recall that $M_j$ is diffeomorphic to $\mathbb{S}^3$ and equipped with a smooth rotationally symmetric Riemannian metric $g_j = ds^2 + f_j(s)^2g_{s^2}$, where $0\leq s\leq D_{j}$, $f_j>0$ on $(0,D_{j})$, $f_j(0)=f_j(D_{j})=0$, $f_j'(0)=1$, and $f_j'(D_{j})=-1$, along with the bounds \be \diam (M_j) \leq D, \ee \be \Scal_j \geq 0,\ee and \be \mina(M_j) \geq A>0. \ee

\subsection{The Upper Bound on Volume}

\begin{lem} \label{lem-Vol}
$\vol(M_j)=4\pi\|f_j\|_{L^2([0, D_j])}^2$
\end{lem}

\begin{proof}
The volume is given by \be \nonumber \ds{\vol(M_j)=\omega_2 \int_0^{D_j} f_j^2(s) \ ds = \omega_2 \|f_j\|^2_{L^2([0, D_j])}},\ee where $\omega_2=4\pi$ is the volume of the unit sphere $\mathbb{S}^2$.
\end{proof}

\begin{lem} \label{lem-V-L2}
By extending $f_j$ as 0 on $[D_j, D]$, \be \nonumber \ds{||f_j||_{L_2([0, D])} \le \sqrt{\frac{V}{\omega_2}}},\ee and a subsequence of the $f_j$ converges to some $\ds{f\in L^2([0,D])}$ weakly.
\end{lem}

\begin{proof}
By Lemma $\ref{lem-Vol}$, $\vol(M_j)\leq V$ implies a uniform bound on the $L^2$ norm of $f_j$,
\be
\|f_j\|_{L^2([0,D])}=\left(\int_0^{D} f_j^2\right)^{1/2}=\left(\int_0^{D_j} f_j^2\right)^{1/2}\leq\sqrt{\frac{V}{\omega_2}},
\ee
and thus, a subsequence of $f_j$ converges to some $f\in L^2([0,D])$ weakly.
\end{proof}

\subsection{The Upper Bound on Diameter}

\begin{lem}\label{lem-diam}
$\diam(M_j)=D_{j}\le D.$
\end{lem}

\begin{proof}
Let $d$ denote the distance function on $M_j$, and $N$, $S$ the north pole (corresponding to $s=0$) and south pole (corresponding to $s=D_j$)  respectively. First we check that $d(N, S)=D_j$. Fix $\theta\in\mathbb{S}^2$ and let $\gamma: [0, D_j]\rightarrow[0, D_j]\times M_j$ be a path defined as $\gamma(t)=(t, \theta)$. Then $\gamma$ is a path connecting $N$ and $S$, and has length $D_j$. Let $\delta: [a, b]\rightarrow [0, D_j]\times\mathbb{S}^2$, $\delta(t)=(s(t), \theta(t))$ be an arbitrary path connecting $N$ and $S$, that is, $\delta(a)=N$ and $\delta(b)=S$. Let $L_j(\delta)$ be the length of $\delta$. Then we have $L_j(\delta)\ge D_j$. Indeed,
\be
\begin{aligned}
L_j(\delta) &=\int^{b}_{a}\sqrt{|s^{\prime}(t)|^2+f^2_j(s(t))|\theta^{\prime}(t)|^2_{g_{\mathbb{S}^2}}} \ dt\\
            &\ge\int^{b}_{a}\left|s^{\prime}(t)\right| \ dt\\
            &\ge\left|\int^b_a s^{\prime}(t) \ dt\right|\\
            &=\left|s(b)-s(a)\right|=D_j.
\end{aligned}
\ee
Because $\delta$ is arbitrary, we obtain $d(N, S)=D_j$.

For any $p,q\in M_j$, let $d(p,N)=d_1$, $d(p,S)=d_2$, $d(q,N)=d_3$, and $d(q,S)=d_4$. Then $d_1+d_2 =d_3+d_4=d(N,S)=D_{j}$. We observe that $$d(p,q)\leq d(p,N)+d(N,q)=d_1+d_3,$$ and similarly, $d(p,q)\leq d_2+d_4$. Since $(d_1+d_3)+(d_2+d_4) = (d_1+d_2)+(d_3+d_4)=2D_{j}$, it follows that either $(d_1+d_3)$ or $(d_2+d_4)$ is at most $D_{j}$. Therefore, $d(p,q)\leq D_{j}$. Thus $\diam(M_j)=D_j$, and so $D_j\le D$, since $\diam(M_j)\le D$.

\end{proof}

\subsection{Scalar Curvature Bounded Below}

\begin{lem}\label{lem-scal}
$\ds{\Scal_j=-\frac{4 f_j ''}{f_j} + \frac{2(1-(f_j')^2)}{f_j^2}.}$
\end{lem}

\begin{proof}
The scalar curvature of the metric $ds^2+f(s)^2 g_{\mathbb{S}^{n-1}}$ is given by $$\Scal = -2(n-1)\frac{f_j ''}{f_j}+(n-1)(n-2)\frac{1-(f_j')^2}{f_j^2}$$ (see \cite{LeeSormani1} or \cite{Pet16}, section 3.2.3). In our case we have $n=3$.
\end{proof}

\begin{lem} \label{lem-hj}
$\Scal_j\ge 0$ is equivalent to
\be\label{hj}
h_j''\le \frac{3}{4}h_j^{-1/3},
\ee
where $h_j(s)=f_j^{3/2}(s)$.
\end{lem}

\begin{proof}
The lemma follows from substituting $f_j = h_j^{2/3}$. Then we have \be \nonumber f_j ' = (2/3)h_j^{-1/3} h_j''\ee and \be \nonumber f_j''=-(2/9)h_j^{-4/3}(h_j')^2+(2/3)h_j^{-1/3}h_j''. \ee
Substituting into Lemma \ref{lem-scal} and simplifying gives the desired result. \end{proof}

\subsection{Lipschitz Bound from Nonnegative Scalar Curvature}

\begin{lem} \label{newlem}
If $\Scal_j \geq 0$ then $|f_j'|\leq 1$ on $(0,D_{j})$.	
\end{lem}

\begin{proof}
Suppose to the contrary that $|f_j'(x_0)|>1$ for some $x_0 \in (0,D_j)$. We may assume that $f_j'(x_0)>1$. Let $x_1 = \inf \{y \in (0,x_0] \mid f_j'(x) > 1 \ \forall x\in (y,x_0]\}$. Then  $x_1<x_0$ and $f_j'(x_1) = 1$ by definition of $x_1$ (if $x_1 = 0$, we understand that $f_j'(0) = 1$ from smoothness of the metric). Since $f_j'$ is $\mathcal{C}^\infty$ on $(0,D_j)$ and continuous at 0, by the mean value theorem there exists $x_2 \in (x_1, x_0)$ such that $f_j''(x_2)(x_0-x_1)=f_j'(x_0) - f_j'(x_1)>0$, where the last inequality follows from that $f_j'(x_0)>1=f_j'(x_1)$. On the other hand, $x_1<x_2<x_0$ implies that $f_j'(x_2)>1$. Therefore $1-(f_j'(x_2))^2<0$. By Lemma \ref{lem-scal}, $\Scal_j\geq 0$ implies that \be \nonumber f_j''(x_2)(x_0-x_1) \leq \frac{1-(f_j'(x_2))^2}{2f_j(x_2)}(x_0-x_1) <0. \ee Hence we have that $0<f_j''(x_2)(x_0-x_1)<0$, a contradiction. Therefore $f_j'\leq 1$ on $(0,D_j)$. A similar argument gives that $f_j'\geq -1$.
\end{proof}

\subsection{The Minimum Area of Minimal Surfaces Bounded Below}

\begin{lem} \label{lem-min-surf}
If $f_j'(s_0)=0$, then $\{s=s_0\}$ is a minimal surface.
\end{lem}

\begin{proof}
Define $\Sigma_s$ as the level set of the coordinate function $s$. Then for all $s\in (0,D_j)$, $\Sigma_s$ is an embeded submanifold with mean curvature
\be H_j=\frac{2f'_j(s)}{f_j(s)}.\ee
Therefore $H_j(s)=0$ if and only if $\Sigma_s$ is minimal.
\end{proof}

\begin{defn}\label{defn-AjBj}
Define $0 < A_j \le B_j < D_j$ as
\be
A_j=\sup\{s \mid \, f_j \textrm{ is increasing on }[0,s]\, \},
\ee
\be
B_j=\inf\{s\mid \, f_j \textrm{ is decreasing on }[s,D_j]\, \}.
\ee
\end{defn}

\begin{lem}\label{lem-AjBj} $f_j'(A_j)=0$ and $f_j'(B_j)=0$. Moreover, $4\pi f_j^2(A_j)\ge \mina(M_j)$ and $4\pi f_j^2(B_j)\ge \mina(M_j)$.
\end{lem}

\begin{proof}
Suppose $f'_j(A_j)>0$. Then $f_j'(s)>0$ for $s\in [A_j-\varepsilon,A_j+\varepsilon]$ for some $\varepsilon > 0 $. Hence $f_j$ is increasing on the interval $[0,A_j+\varepsilon]$, a contradiction. Similarly, $f_j'(A_j)<0$ cannot hold, which proves that $f'_j(A_j)=0$. By the same argument $f_j'(B_j)=0$. By Lemma $\ref{lem-min-surf}$ there is a minimal surface at $s=A_j$ and $s=B_j$, which have areas $4\pi f_j^2(A_j)$ and $4\pi f_j^2(B_j)$ respectively.
\end{proof}

\begin{lem}\label{lem-middle-minA}
If $A_j<B_j$, then $4\pi f_j^2(s) \ge \mina(M_j)$ for all $s\in [A_j,B_j]$.
\end{lem}
\begin{proof}
If $A_j<B_j$, then by continuity there exists $s_0\in [A_j,B_j]$ such that $f_j(s)\geq f_j(s_0)$ for all $s\in [A_j,B_j]$. By Lemma $\ref{lem-min-surf}$ there is a minimal surface at $s=s_0$. As a result, by definition of $\mina(M_j)$, we have
\be
4\pi f_j^2(s)\geq 4\pi f_j^2(s_0)\geq \mina(M_j)
\ee
for all $s\in [A_j,B_j]$.
\end{proof}

\begin{lem}\label{lem-Dj-bounds}
\be
\frac{A}{4\pi}\le D_j^2 \le D^2.
\ee
\end{lem}
\begin{proof}
In Lemma~\ref{lem-diam} we have obtained $D_j\le D$. On the other hand, from Lemma~\ref{newlem} and the definition of $A_{j}$, we have $f_{j}(A_j)\le A_j\le D_j^2$. Combining this with Lemma~\ref{lem-AjBj} gives $\ds{\frac{A}{4\pi}\le D_j^2}$.
\end{proof}

\section{\bf Examples}
\begin{example}\label{collapsing}
{\rm We will construct a family of 3-dimensional smooth closed rotationally symmetric Riemannian manifolds $M_{j}$, which are isometric to three-spheres with the Riemannian metrics $g_{j}=ds^{2}+f^{2}_{j}(s)g_{\mathbb{S}^{2}},$
where $s\in[0, D_j]$, such that
\be\label{Ex1-1}
\Scal_{j} \leq 0, \quad \textit{\rm for all}\ \ j\in\mathbb{N},
\ee
\be\label{Ex1-2}
\mina(M_{j})\rightarrow0, \quad \textit{\rm as} \ \ j\rightarrow\infty,
\ee
and
\be\label{Ex1-3}
M_{j}\Fto {\bf 0}, \quad \textit{\rm as} \ \ j\rightarrow\infty,
\ee
where ${\bf 0}$ is the zero current, since $\vol(M_{j})\rightarrow0$ as $j\rightarrow\infty$.

Let $\phi_{j}$ be a sequence of smooth functions defined on $[0, D]$ satisfying:

 (a) $\phi_{j}(0)=1$, and $\phi_{j}(D)=-1$;

 (b) $\phi_{j}$ is monotone non-increasing; that is, $\phi^{\prime}_{j}\leq 0$;

 (c) $\phi_{j}$ is symmetric about the point $(D/2, 0)$; that is, $\phi_{j}(s)=-\phi_{j}(D-s)$ for all $s\in D$;

 (d) $\ds{\lim\limits_{j\rightarrow\infty}\int^{\frac{D}{2}}_{0}\phi_{j}(s) \ ds=0}$.

Define functions $f_{j}$ on $[0, D]$ as
\be
f_{j}(s):=\int^{s}_{0}\phi_{j}(t) \ dt.
\ee

For example, we can set $D=2$, and
\be\label{Ex1-4}
\phi_{j}(s)=(1-s)^{2j+1}, \quad \text{defined on} \ \ [0, 2].
\ee
Then $\ds{f_{j}=\frac{1}{2j+2}-\frac{(1-s)^{2j+2}}{2j+2}}$ on [0, 2].

From the above properties of $\phi_{j}$, we have

 (a$^{\prime}$) $f^{\prime}_{j}(0)=1$, and $f^{\prime}_{j}(D)=-1$;

 (b$^{\prime}$) $f^{\prime\prime}_{j}=\phi^{\prime}_{j}\leq0$, and $|f^{\prime}_{j}|=|\phi_{j}|\leq 1$;

 (c$^{\prime}$) $f_{j}\geq 0$ with $f_{j}(0)=f_{j}(D)=0$ and $f_{j}>0$ everywhere else;

 (d$^{\prime}$) $\max\limits_{s\in [0, D]}\{f_{j}(s)\}=f_{j}(D/2)\rightarrow 0$ as $j\rightarrow \infty$.

By (a$^{\prime}$) and (c$^{\prime}$) above, $g_{j}=ds^{2}+f^{2}_{j}(s)g_{\mathbb{S}^{2}}$ is a smooth Riemannian metric on $\mathbb{S}^3$.

By (b$^{\prime}$), $g_{j}$ have nonnegative scalar curvatures. Indeed, (b$^{\prime}$) implies
$$
2f_{j}(s)f^{\prime\prime}_{j}(s)\leq 0\leq 1-(f^{\prime}_{j}(s))^{2}
$$
for all $s\in[0, D]$. This further implies
$$
\Scal_{j}=-\frac{4f^{\prime\prime}_{j}}{f_{j}}+\frac{2(1-(f^{\prime}_{j})^{2})}{f^{2}_{j}}\geq0.
$$

Finally, by (d$^{\prime}$), we have
$$
\mina(M_{j})\leq\vol(\{D/2\}\times\mathbb{S}^{2})=4\pi f^{2}_{j}(D/2)\rightarrow0,
$$
and
$$
\vol(M_{j})=\int^{D}_{0}\int_{\mathbb{S}^{2}}f^{2}_{j}(s) \, d{\rm vol}_{g_{\mathbb{S}^{2}}}=4\pi\int^{D}_{0}f^{2}_{j}(s) \, ds\leq4\pi f_{j}(D/2)D\rightarrow0,
$$
as $j\rightarrow\infty$.

}
\end{example}

\begin{example}[Example 5.9 in \cite{Lak16}] \label{lakzian}
{\rm
In Example 5.9 in \cite{Lak16}, Lakzian has shown that there are metrics $g_{j}$ on the sphere $\mathbb{S}^{3}$ with positive scalar curvature such that the family of rotationally symmetric Riemannian manifolds $M_{j}=(\mathbb{S}^{3}, g_{j})$ has the SWIF limit round sphere $\mathbb{S}^{3}$, and the Gromov-Hausdorff limit $\mathbb{S}^{3}\sqcup[0, 1]$, the round sphere $\mathbb{S}^{3}$ with an interval of length $1$ attached to it. Actually, these $M_{j}$ satisfy all hypotheses in Theorem~\ref{main-result}. Lakzian has shown that $\Scal_{j}>0$ and $\diam(M_{j})\leq \pi+3$. Moreover, one can easily check that $\mina(M_{j})=4\pi$. Now we briefly recall Lakzian's examples. They are $\mathbb{S}^3$ with a spline of finite length and arbitrary small width attached to it, and have positive scalar curvature. For fixed $L$ (the length of the spline will be between $L$ and $L+2$) and $\delta<1$ ($\delta$ will be width of the spline), let $m_{H}(r)$ be an admissible Hawking mass function (which has to be smooth and increasing) that satisfies
\be
m_{H}(r)=\frac{r(1-\varepsilon^2)}{2}, \quad \text{\rm for} \ \ r\in [0, \delta^{3}],
\ee
and
\be
m_{H}(r)=\frac{r^3}{2}, \quad \text{\rm for} \ \ r\in[\delta, 1].
\ee
where $\varepsilon$ is chosen so that
\be
\delta^3\sqrt{\frac{1-\varepsilon^2}{\varepsilon^2}}=L.
\ee
Then define
\be
z^{\prime}(r)=\sqrt{\frac{2m_H(r)}{r-2m_H(r)}}.
\ee
Note that $z^{\prime}(r)$ depends on $\delta$. So it will be denoted by $z^{\prime}_\delta(r)$.

Define the rotationally symmetric metric $g_{\delta}$ on $\mathbb{S}^3=[0, \pi]\times\mathbb{S}^2$ to be
\be
(1+[z^{\prime}_{\delta}(\sin(\rho))]^2)\cos^2(\rho)d\rho^2+\sin^2(\rho)g_{\mathbb{S}^2} \quad \text{for}\ \ \rho\in[0, \pi/2],
\ee
and
\be
d\rho^2+\sin^2(\rho)g_{\mathbb{S}^2} \quad \text{for} \ \ \rho\in[\pi/2, \pi].
\ee
By doing a certain implicit change of variable the metric on the part of $\rho\in[0, \pi/2]$ can be written as  $\rho = ds^2+f^2_\delta(s)g_{\mathbb{S}^2}$. For $\delta_{j}>0$ with $\delta_{j}\rightarrow0$ as $j\rightarrow\infty$ and suitable choice of $L$, $M_j=(\mathbb{S}^3, g_{\delta_j})$ give the example. \rm For more details about this example we refer to \cite{Lak16}.
}

\end{example}

\begin{example}[Example 3.12 in \cite{AS18}] \label{non-euc}
{\rm
In Example 3.12 in \cite{AS18}, Allen and Sormani construct a sequence of warped product metrics on $\mathbb{S}^1\times \mathbb{S}^2$p where the warping functions converge to 1 on a dense set. However, the metrics converge in Gromov-Hausdorff and SWIF sense to a metric space which is not a Riemannian manifold. In fact, no local tangent cone on this limit is isometric to the Euclidean space. It is possible to construct warped product metrics on $\mathbb{S}^3$ by cutting $\mathbb{S}^1$ to get an interval and capping off with hemispheres. Then the tangent cones in the middle region are not Euclidean. For details we refer to \cite{AS18}.
}
\end{example}

\section{\bf Properties of the Limit Space}

In this section, we will define the limit space with the continuous limit metric, and show that it has Euclidean tangent cones almost everywhere. We will also show that the metric is $H^1$ and has nonnegative scalar curvature in the sense that it satisfies $(\ref{hj})$ as a distribution.

From now on, we extend the warping functions $f_j$ defined on $[0, D_j]$ to functions defined on $[0, D]$ by setting $f_j=0$ on $[D_j, D]$. Then $f_j$ are continuous on $[0, D]$ and smooth everywhere on $(0, D)$ except at $D_j$.

Take $0<A_j \le B_j<D$ as in Definition~\ref{defn-AjBj}.  There is a subsequence
such that $A_j \to A_\infty$ and $B_j \to B_\infty$ where
\be \label{defn-Ainfty-Binfty}
0\le A_\infty\le B_\infty\le D.
\ee
\begin{thm} \label{limfun}
A subsequence of $f_j$ converge uniformly to a Lipschitz function $f_\infty$ on $[0,D]$, which has Lipschitz constant 1 and satisfies the following properties.

({\rm i}) $f_\infty(0) = 0$ and $f_\infty$ is nondecreasing on $[0,A_\infty]$,

({\rm ii}) $f_\infty \geq \sqrt{A/4\pi}$ on $[A_\infty,B_\infty]$ if $A_{\infty}\neq B_{\infty}$,

({\rm iii}) $f_\infty(D) = 0$ and $f_\infty$ is nonincreasing on $[B_\infty,D]$.

\end{thm}

\begin{proof}
By Lemma \ref{newlem} all functions $f_{j}$ are Lipschitz with Lipschitz constant 1 on the interval $[0, D]$. Indeed, take any $x<y\in[0, D]$. If $x, y\in[0, D_j]$, then Lemma \ref{newlem} implies $|f_j(x)-f_j(y)|\le|x-y|$. If $x, y\in[D_j, D]$, then $f_j(x)=f_j(y)=0$ so $|f_j(x)-f_j(y)|\le|x-y|$. Finally if $x\le D_j<y\le D$, then since $f_j(D_j)=f_j(y)=0$, we have
$$
\frac{|f_j(x)-f_j(y)|}{|x-y|}=\frac{|f(x)-f(D_j)|}{|x-y|}\le\frac{|f(x)-f(D_j)|}{|x-D_j|}\le1.
$$
By combining with Arzel\`{a}-Ascoli theorem we obtain the uniform convergence. (ii) is then immediate from Lemma \ref{lem-middle-minA}. (i) and (iii) follows from the monotonicity of $f_j$ on $\ds{[0,A_j]\cup [B_j,D]}$.
\end{proof}

\begin{lem}\label{ak-bk}
Given sufficiently large $k>0$, the set
\be
I_k=\left\{x \, \bigg| \,f_\infty(x) \ge \frac{1}{k}\right\}
\ee
is a connected interval, $I_k=[a_k,b_k]$.
\end{lem}

\begin{proof}
$I_k$ is closed since $f_\infty$ is continuous. If $A_\infty \neq B_\infty$, by Lemma \ref{lem-middle-minA}, $f_j > \sqrt{A/4\pi}$ on $[A_j, B_j]$ for all $j$. Since $A_j \to A_\infty$ and $B_j \to B_\infty$ as $j\to \infty$, if $j$ is large then $f_j > \sqrt{A/16\pi}$ on $[A_\infty, B_\infty]$. Take $k$ large enough so that \be \nonumber \frac{1}{k} \leq \min \left\{f_\infty(A_\infty), f_\infty(B_\infty), \sqrt{\frac{A}{16\pi}}\right\}. \ee If $A_{\infty}=B_{\infty}$, then take $k$ large enough so that $\frac{1}{k} \leq f_{\infty}(A_{\infty})$. Then by Theorem \ref{limfun}, $I_k$ is a connected interval containing $[A_\infty,B_\infty]$.
\end{proof}

Note that $\ds{f_\infty(a_k) = f_\infty(b_k) = \frac{1}{k}}$. We set
\be
a_\infty:=\sup\{s\mid\, f_\infty(t)=0 \textrm{ on } [0,s] \,\} \in [0, A_\infty]
\ee
and
\be
b_\infty:=\inf\{s\mid \, f_\infty(t)=0 \textrm{ on } [s,D] \,\} \in [B_\infty, D].
\ee

Then we immediately have the following lemma.

\begin{lem}\label{a-b}
Let $a= \inf\{a_k\mid \, k>0\}$ and $b=\sup\{b_k\mid \, k>0\}$ so that
\be
(a,b) = \bigcup_{k>0} I_k.
\ee
Then $(a,b)=\{x:\, f_\infty(x)>0\}$ so $a=a_\infty$ and $b=b_\infty$.
\end{lem}

\begin{prop} \label{finfty-to-zero}
$f_\infty(a) = f_\infty(b) =0$.
\end{prop}

\begin{proof}
Since $f_\infty$ is continuous, \be \nonumber f_\infty(a) = \underset{k\to\infty}{\lim}f_\infty(a_k) = \underset{k\to\infty}{\lim}\frac{1}{k} =0.\ee Similarly $f_\infty(b) =0$.
\end{proof}

\begin{defn}\label{limitspace}
The limit space
\be
M_\infty=[a_\infty, b_\infty] \times {\mathbb{S}}^2
\ee is
a warped product Riemannian manifold, diffeomorphic with $\mathbb{S}^3$, with the continuous metric
tensor
\be
g_\infty=ds^2 + f_\infty^2(s) g_{{\mathbb{S}}^2}.
\ee
\end{defn}

\begin{thm}\label{tangcone}
The local tangent cones of $M_\infty$ are $\mathbb{E}^3$ almost everywhere.	
\end{thm}

\begin{proof}
Since $f_\infty$ is Lipschitz, it follows that $f_\infty$ is differentiable almost everywhere on $[a_\infty, b_\infty]$; that is, the limit $a_p = \displaystyle{\lim_{s\to s_p} \frac{f_\infty(s_i)-f_\infty(s_p)}{s_i - s_p}}$ exists at almost every $p=(s_p,\theta_p)$. Let $l(s) = f(s_p) + a_p(s-s_p)$ be the linear function that best approximates $f(s)$ at $s=s_p$. Then the tangent cone of $M_\infty$ at $p$ is the warped product $ds^2+l(s)^2g_{\mathbb{S}^2}$, which is isometric to the Euclidean space.	
\end{proof}

\begin{rmrk}
{\rm
The scalar curvature control, which in turn gave Lipschitzness of $f_\infty$, is of crucial importance in this argument; compare with Example \ref{non-euc}.}
\end{rmrk}

\begin{thm} \label{H-1-loc}
The sequence $h_j=f_j^{3/2}$, after possibly passing to a subsequence, converges in $H^1_{loc}$ to $h_\infty \in H^1(I)$, where $I$ is the open interval $(a_\infty, b_\infty)$. Defining $f_\infty = h_\infty^{2/3}$, a subsequence of $f_j$ also converges in $H^1_{loc}$ to $f_\infty\in H^1(I)$.\end{thm}

\begin{proof}
First we will show that when $k$ is large enough, there is a uniform bound on the variation of $h_j'$, that is,
\be
\sup_j \left\|h_j'\right\|_{BV(I_k)}<\infty.
\ee
By definition of $h_j$ we have
\be
h_j'(s)=\frac{3}{2}f_j^{1/2}(s)f_j'(s),
\ee
for $s\in[0, D_j)$.
By Lemma $\ref{newlem}$ we have $|f_j'(s)|\leq 1$ for all $j$ and for all $s\in [0,D_j]$, hence we have $0\leq f_j(s)\leq D_j/2\le D/2$ for all $j$ and all $s\in [0,D_j]$. As a result, $|h_j'(s)|\le D/2$ for all $s\in[0, D_j)$. Recall that $f_j$ has been extended as $0$ on $[D_j, D]$, so $h_j=0$ on $[D_j, D]$, and $h_j'=0$ on $(D_j,D]$. Thus by the same argument as in proof of Theorem \ref{limfun}, $\{h_j\}$ is uniformly Lipschitz. Then by Arzel\`{a}-Ascoli we have that a subsequence of $h_j$ converges to some $h_\infty$ uniformly in $[0,D]$. The limit function $h_\infty$ is also Lipshitz, and $h_\infty=(f_\infty)^{3/2}$ (since we only need this pointwise). Since a Lipshitz function defined on an interval is actually $W^{1,\infty}$, we have $h_\infty\in W^{1,\infty}(I)$. Since for each large enough $j$, $h_j'$ is smooth on $I_k$, we have
\be
\begin{split}
\left\|h_j'\right\|_{BV(I_k)} &=\int_{I_k} \left|h_j''\right| \ ds\\
&=\int_{\{s\in I_k \mid h_j''(s)\geq 0\}} h_j'' \  ds-\int_{\{s\in I_k \mid h_j''(s)< 0\}} h_j'' \ ds.
\end{split}
\ee

Note that
\be
h_j''=\frac{3}{4}f_j^{-1/2}(s)(f_j'(s))^2+\frac{3}{2}f_j^{1/2}(s)f_j''(s).
\ee

Let $j$ be so large that
\be
\left|h_j(s)-h_\infty (s)\right|<\frac{1}{3}\left(\frac{1}{k}\right)^{3/2}
\ee
for all $s\in (a_k,b_k)$. Then by definition of $I_k$, we have
\be
\frac{2}{3}\left(\frac{1}{k}\right)^{3/2}\leq h_{\infty}(s)-\frac{1}{3}\left(\frac{1}{k}\right)^{3/2}\leq h_j(s)
\ee
for all $s\in (a_k,b_k)$. Moreover, by Lemma $\ref{lem-hj}$, we have
\be
h_j''(s)\leq\frac{3}{4}h_j(s)^{-1/3}\leq \frac{3}{4}\left(\frac{2}{3}\left(\frac{1}{k}\right)^{3/2}\right)^{-1/3}
\ee
for $j$ large enough and for all $s\in I_k$. As a result, we have when $j$ is large,
\be
\int_{\{s\in I_k \mid h_j''(s)\geq 0\}} h_j'' \ ds\leq \frac{3}{4}\left(\frac{2}{3}\left(\frac{1}{k}\right)^{3/2}\right)^{-1/3}(b_k-a_k).
\ee
Moreover, since
\be
\int_{\{s\in I_k \mid h_j''(s)\geq 0\}} h_j'' \ ds+\int_{\{s\in I_k \mid h_j''(s)< 0\}} h_j'' \ ds=\int_{a_k}^{b_k}h_j''(s)\ ds=h_j'(b_k)-h_j'(a_k),
\ee
we have
\be
\begin{split}
\left\|h_j'\right\|_{BV(I_k)}&=\int_{a_k}^{b_k}\left|h_j''\right| \ ds\\
&=2\int_{\{s\in I_k \ \mid \ h_j''(s)\geq 0\}} h_j'' \ ds+h_j'(a_k)-h_j'(b_k)\\
&\leq \frac{3}{2}\left(\frac{2}{3}\left(\frac{1}{k}\right)^{3/2}\right)^{-1/3}(b_k-a_k)+3\left(\frac{D}{2}\right)^{1/2},
\end{split}
\ee
for all $j$ large enough.

As a result, by Theorem 5.5 in \cite{EG15} we have that $h_j'$ converges to some $\phi$ in $L^1 (I_k) $ norm. It is easy to show that $\phi=h_\infty '$ in the weak sense by a density argument. Moreover, since $h_\infty\in W^{1,\infty}(I)$ and
\be
\sup_{j}\left\|h_j'\right\|_{L^\infty(I_k)}<\infty,
\ee
we have $h_j'\to h_\infty '$ in $L^2 (I_k)$ norm. Note that by the H\"{o}lder inequality,
\be
\int_{I_k} \left|h_j'-h_\infty '\right|^2\leq \left\|h_j'-h_\infty '\right\|_{L^1(I_k)} \left\|h_j'-h_\infty '\right\|_{L^\infty(I_k)}.\ee
As a result $h_j\to h_\infty$ in $H^1_{loc}(I_k)$ norm.

Now we turn to the convergence of $f_j$. First note that the function $f(\xi)=\xi^{2/3}$ is $C^1$ with $f'(\xi)$ bounded when $\xi\geq \varepsilon>0$ for some $\varepsilon\in \R$. By the chain rule for weak derivatives we know that the weak derivative of $f_{\infty}$ exists, and that
\be
f_{\infty}'=\frac{2}{3}h_\infty^{-1/3}h_\infty'.
\ee
Since $h_j\to h_\infty$ uniformly on $[0,D]$, we have $\ds{h_j\geq \frac{1}{2}\left(\frac{1}{k}\right)^{3/2}>0}$ on $I_k$ for large $j$. Therefore,
\be
|f_j'-f_\infty'|\leq \frac{2}{3}\left(\frac{1}{2}\left(\frac{1}{k}\right)^{3/2}\right)^{-1/3}|h_j'-h_\infty'|
\ee
for large $j$. Since $h_j\to h_\infty$ in $H^1_{loc}(I)$, it follows that $f_j\to f_\infty$ in $H^1_{loc}(I)$.

\end{proof}

\begin{thm} \label{gen-scalar-mid}
$g_\infty$ has nonnegative generalized scalar curvature on the interior $\mathring{M}_\infty = M_\infty\setminus(\{a_\infty\}\times \mathbb{S}^2\cup\{b_\infty\}\times\mathbb{S}^2)$ of $M_\infty$, in the sense that $f_\infty$ satisfies $(\ref{hj})$ as a distribution on $\mathring{M}_\infty$.
\end{thm}

\begin{proof}
Fix a large $k$. For $j$ large enough, and for any $u\in C^\infty_c(I_k)$ such that $u\geq 0$, by Lemma $\ref{lem-scal}$ after some calculation we have
\be
\int_{I_k} (1+(f_j')^2)u\ ds\geq 2\int_{I_k}(f_j'f_j)'u\ ds.
\ee
After integration by part on the right hand side, we get
\be
\int_{I_k} (1+(f_j')^2)u\ ds\geq -2\int_{I_k}(f_j'f_j)u'\ ds.
\ee
Since
\be
\begin{split}
\left|\int_{I_k}\left(\left(f_j'\right)^2-f_\infty'^2\right)u\ ds\right| &\leq \left\|u\right\|_{L^\infty(I_k)}\cdot \left|\left\|f_j'\right\|^2_{L^2(I_k)}-\left\|f_\infty'\right\|^2_{L^2(I_k)}\right|\\
&\leq \left\|u\right\|_{L^\infty(I_k)}\cdot\left(\left\|f_j'\right\|_{L^2(I_k)}+\left\|f_\infty'\right\|_{L^2(I_k)}\right)\cdot \left|\left\|f_j'\right\|_{L^2(I_k)}-\left\|f_\infty'\right\|_{L^2(I_k)}\right|
\\
&\leq \left\|u\right\|_{L^\infty(I_k)} \cdot \left(\left\|f_j'\right\|_{L^2(I_k)}+\left\|f_\infty'\right\|_{L^2(I_k)}\right)\cdot\left\|f_j'-f_\infty'\right\|_{L^2(I_k)},
\end{split}
\ee
and

\be
\begin{split}
\left|\int_{I_k}\left(f_j' f_j-f_\infty' f_\infty\right) u'\ ds\right| &\leq \int_{I_k}\left|f_j'f_j-f_j' f_\infty\right| \cdot \left|u'\right|\ ds+\int_{I_k}\left|f_j'f_\infty-f_\infty' f_\infty\right| \cdot \left|u'\right|\ ds\\
&\leq \left\|f_j'\right\|_{L^2(I_k)} \cdot \left\|f_j-f_\infty\right\|_{L^2(I_k)} \cdot \left\|u'\right\|_{L^\infty(I_k)}+ \left\|f_\infty\right\|_{L^2(I_k)} \cdot \left\|f_j'-f_\infty'\right\|_{L^2(I_k)} \cdot \left\|u'\right\|_{L^\infty(I_k)},
\end{split}
\ee
by Theorem \ref{H-1-loc} and density, we have
\be
\int_{I_k} \left(1+\left(f_\infty'\right)^2\right)u\ ds\geq -2\int_{I_k}\left(f_\infty'f_\infty\right)u'\ ds.
\ee

Which means for any $u\in C_c^\infty (I)$ with $u\geq 0$, we have

\be
\int_{I} \left(1+\left(f_\infty'\right)^2\right)u\ ds\geq 2\int_{I}\left(f_\infty'f_\infty\right)'u\ ds,
\ee
where we think of $(f_\infty' f_\infty)'$ as a distribution on $I$. Define $\tilde{u}$ on $\mathring{M}_\infty$ by
\be
\tilde{u}(s,\theta):=u(s).
\ee
Then by the previous argument we have
\be
\int_{\mathring{M}_\infty} \Scal_\infty \tilde{u}\ d \text{vol}_\infty\geq 0
\ee
in the sense of distribution. Here $\Scal_\infty=\frac{-4 (f_\infty 'f_\infty)'+2(1+(f_\infty')^2)}{f_\infty^2}$ is viewed as a distribution on $\mathring{M}_\infty$, and $d\text{vol}_\infty=f_{\infty}^2 ds d\text{vol}_{{\mathbb{S}}^2}$ is the volume form on $\mathring{M}_\infty$.

For a general $\tilde{v}\in C_c^\infty (\mathring{M}_{\infty})$ with $\tilde{v}\ge 0$, define
\be
v(s):=\int_{\mathbb{S}^2}\tilde{v}(s,\theta)\ d\theta.
\ee
Since $\mathbb{S}^2$ is compact, differentiation by $s$ commutes with integration. As a result, $v\in C_c^\infty (I)$. By the previous argument we have
\be
\int_{\mathring{M}_\infty} \Scal_\infty \tilde{v}\ d \text{vol}_\infty \geq 0
\ee
in the sense of distribution.
\end{proof}

\br
{\rm
When $\Scal_j\ge 0$ is replaced by $\Scal_j\ge H>0$, as mentioned in Remark $\ref{Scal-ge-6}$, we can still use $\Scal_j\ge H>0$ to get uniform convergence to a Lipshitz function (as in Lemma $\ref{newlem}$ anad Theorem $\ref{limfun}$) and $H^1_{loc}$ convergence (as in Theorem $\ref{H-1-loc}$). Then we can use a similar argument as in Theorem $\ref{gen-scalar-mid}$ to show $\Scal_\infty\ge H>0$ in the sense of distribution.
}
\er

In Theorem \ref{gen-scalar-mid} we have obtained nonnegativity of scalar curvature of $g_\infty$ restricted on $\mathring{M}_\infty$ in the sense of distributions. Now we consider generalized scalar curvature in the sense of small volumes at the two poles of $M_\infty$, which are $p_{a_\infty}=(a_\infty, \theta)$ and $p_{b_\infty}=(b_\infty,\theta)$. Recall that on a 3-dimensional smooth Riemannian manifold $(M, g)$ the scalar curvature at a point $p\in M$ can be expressed as
\be
\Scal_g(p)=30\cdot\lim_{r\rightarrow0}\frac{\frac{4\pi}{3}r^3 - \vol(B(p, r))}{\frac{4\pi}{3}r^5},
\ee
where $B(p, r)$ is the ball in $M$ centered at $p$ of radius $r$. Thus we will show that $g_\infty$ has nonnegative generalized scalar curvature at points $p_{a_\infty}$ and $p_{b_\infty}$ in the sense of satisfying the following inequalities.

\begin{prop} \label{voltip}
The limit metric $g_\infty$ stasties
\be\label{tip1}
\liminf_{r\rightarrow0}\frac{\frac{4\pi}{3}r^3 - \vol(B(p_{a_\infty}, r))}{\frac{4\pi}{3}r^5}\ge0,
\ee
and 
\be\label{tip2}
\liminf_{r\rightarrow0}\frac{\frac{4\pi}{3}r^3 - \vol(B(p_{b_\infty}, r))}{\frac{4\pi}{3}r^5}\ge0.
\ee
\end{prop}

\begin{proof}
We will prove the inequality (\ref{tip1}). Using polar coordinates, 
\be
\vol(B(p_{a_\infty}, r)) = \int_a^{r+a} f_\infty(s)^2 \text{Area}(\mathbb{S}^2) \, ds = 4\pi \int_a^{r+a} f_\infty^2(s)ds.
\ee
Therefore, to prove (\ref{tip1}) it suffices to show
\be
\liminf_{r\rightarrow0} \frac{\frac{4\pi}{3} r^3 -  4\pi \int_a^{r+a} f_\infty^2(s)ds}{\frac{4\pi}{3} r^5}\geq 0.
\ee
Note that this limit can be written as 
\be
\liminf_{r\rightarrow0} \frac{3\int_a^{r+a} ((s-a)^2-f_\infty^2(s))ds}{r^5}.
\ee

We claim $0\le f_\infty (s) \leq s-a$ for $s\in [a,a+\epsilon]$ for small $\epsilon>0$, from which (\ref{tip1}) would follow. If on the contrary we had $f_\infty(r_0) > r_0-a$ for some $r_0 >a$, there exists $k>0$ such that $f_\infty(r_0) > r_0 - a + \frac{1}{k}$. Then if $j$ is so large that $\left\|f_\infty - f_j\right\|_{\mathcal{C}^0([0,D])} \leq \frac{1}{2k}$, it follows that
\begin{align*}
f_j(r_0) &\ge f_\infty(r_0) - 	\left\|f_\infty - f_j\right\|_{\mathcal{C}^0([0,D])} \\
&> r_0 - a +\frac{1}{k} - \left\|f_\infty - f_j\right\|_{\mathcal{C}^0([0,D])} \\
&> r_0 - a + \frac{1}{2k}.
\end{align*}
Therefore, 
\be
\frac{f_j(r_0) - f_j(a)}{r_0 - a} > \frac{\frac{1}{2k} - f_j(a)}{r_0 -a} +1.
\ee
On the other hand, we have 
\be
1\geq \frac{f_j(r_0) - f_j(a)}{r_0 - a}
\ee
by Lemma \ref{newlem}. Combining these together, it follows that 
\be
f_j(a) > \frac{1}{2k}
\ee 
for all large enough $j$.
By taking $j \to \infty$, it follows
$$\frac{1}{2k}\le\lim_{j\rightarrow\infty}f_j(a) = f_\infty(a) = 0,$$ 
a contradiction. Therefore $f_\infty(s) \leq s-a$. The inequality (\ref{tip2}) can be shown similarly.
\end{proof}

\begin{rmrk}
{\rm Whether Proposition \ref{voltip} is true everywhere on $M_\infty$ is an interesting question that we have not been able to answer so far. If it is true, it gives another way of generalizing nonnegativity of scalar curvature to the possibly singular space $M_\infty$, as ``small infinitesimal volumes".}
\end{rmrk}

\section{\bf Intrinsic Flat Convergence to the Limit}

In this section we will prove that there exists a subsequence of $M_j$ that converges to $M_\infty$ in the sense of the SWIF distance.

Recall that in Theorem~\ref{limfun}
we obtained the uniform convergence (possibly passing to a subsequence)
\be
f_j \to f_\infty \textrm{ on } I_k \subset (a,b)\subset [0,D].
\ee

For each $k>0$ we define the following sets
\begin{eqnarray}
W_j&=& W_j^k :=\{ (s,\theta)\in M_j:\, s\in I_k,\, \theta\in {\mathbb{S}}^2\} \subset M_j, \\
W_\infty&=& W_\infty^k :=\{ (s,\theta)\in M_\infty:\, s\in I_k,\, \theta\in {\mathbb{S}}^2\} \subset M_{\infty},
\end{eqnarray}
which are diffeomorphic to $W:= I_k \times {\mathbb{S}}^2$ by diffeomorphisms
\be
\psi_{j}: W \rightarrow W_{j} \quad \textit{\rm and} \quad \psi_{\infty}: W\rightarrow W_{\infty},
\ee
with metric tensors induced from $M_{j}$ and $M_{\infty}$, respectively.
We will use the uniform convergence of the metric tensors $g_j \to g_\infty$ on $W$
to prove the SWIF convergence.  To do so, we will apply the following theorem
of Lakzian-Sormani  \cite{Lakzian-Sormani}:

\begin{thm}[Theorem 4.6 in \cite{Lakzian-Sormani}] \label{thm-subdiffeo}
Suppose $(M_j,g_j)$ and $(M_\infty,g_\infty)$ are oriented
precompact Riemannian manifolds
with diffeomorphic subregions $W_j \subset M_j$ and $W_\infty \subset M_\infty$.
Identify $W_j=W_\infty=W$ by diffeomorphisms $\psi_{j}: W\rightarrow W_{j}$ and $\psi_{\infty}: W\rightarrow W_{\infty}$. Assume that
on $W$ with induced metrics by $\psi_{j}$ and $\psi_{\infty}$ we have
\be \label{thm-subdiffeo-1}
g_j \le (1+\varepsilon)^2 g_\infty \quad \textrm{ and } \quad
g_\infty \le (1+\varepsilon)^2 g_j.
\ee
Then the SWIF distance satisfies
\begin{eqnarray*}
d_{\mathcal{F}}(M_j, M_\infty)
&\le&
\left(2\bar{h} + a\right) \Big[
\vol(W_{j})+\vol(W_\infty)+\area(\partial W_{j})\\
& &+\area(\partial W_{\infty})\Big]
+\vol(M_j\setminus W_j)+\vol(M_\infty\setminus W_\infty),
\end{eqnarray*}
where
\be \label{thm-subdiffeo-5}
\bar{h}= \max\{h,  D_0\sqrt{\varepsilon^2 + 2\varepsilon} \}.
\ee
Here, $a$, $h$, and $D_0$ are defined as follows:
\be
\max\{\diam(M_j),\diam(M_\infty)\} \le D_0,
\ee
\be \label{lambda}
\lambda:=\sup_{x,y \in W}
|d_{M_j}(\psi_j(x),\psi_j(y))-d_{M_\infty}(\psi_\infty(x),\psi_\infty(y))|,
\ee
\be \label{thm-subdiffeo-4}
h =\sqrt{\lambda ( D_0 +\lambda/4 )\,} \le \sqrt{2\lambda D_0},
\ee
\be \label{thm-subdiffeo-3}
a\geq \frac{\arccos(1+\varepsilon)^{-1} }{\pi}D_0.
\ee
\end{thm}

First note that by the uniform convergence proven in Theorem~\ref{limfun}
for any $k>0$ we can take $j$ large enough to have (\ref{thm-subdiffeo-1}).

\subsection{Small Volumes}
Next we show that the volumes of $M_j\setminus W_j$ are small.

\begin{lem} \label{lem-vol-M-not-W}
For each fixed large $k>0$, if $j$ is sufficiently large then the following bounds hold.
\be
\vol(M_j\setminus W_j)\le  \frac{16\pi D}{k^{2}},
\ee
\be
\vol(M_\infty\setminus W_\infty)\le \frac{4\pi D}{k^{2}}.
\ee
\end{lem}

\begin{proof}
We choose and fix a large $k>0$ so that $I_{k}$ defined in Lemma $\ref{ak-bk}$ is a connected interval, $I_{k}=[a_{k}, b_{k}]$. Then,
$$
M_{j}\setminus W_{j}=\{(s, \theta) \mid s\in [0, a_{k})\cup (b_{k}, D], \ \ \theta\in \mathbb{S}^{2}\} \subset M_{j}.
$$

Because $f_{j}$ converges to $f_{\infty}$ uniformly on $[0, D]$ (by passing to a subsequence, if necessary), there exists a large $j_{0}$ such that for all $j>j_{0}$,
$$
|f_{j}(s)-f_{\infty}(s)|<\frac{1}{k}, \qquad \forall s\in [0, D].
$$
In particular, for any $j>j_{0}$ and any $s\in [0, a_{k})\cup (b_{k}, D]$, we have
$$
f_{j}(s)<f_{\infty}(s)+\frac{1}{k}<\frac{2}{k}.
$$

Thus,
\begin{align*}
\vol(M_{j}\setminus W_{j})
&=\int^{a_{k}}_{0}\int_{\mathbb{S}^{2}}f^{2}_{j}(s) \ d{\rm vol}_{g_{\mathbb{S}^{2}}} ds+\int^{D}_{b_{k}}\int_{\mathbb{S}^{2}}f^{2}_{j}(s) \ d{\rm vol}_{g_{S^{2}}}ds\\
&=4\pi\int^{a_{k}}_{0}f^{2}_{j}(s) \ ds+4\pi\int^{D}_{b_{k}}f^{2}_{j}(s)\ ds\\
&<4\pi \frac{4}{k^{2}}a_{k} +4\pi \frac{4}{k^{2}}(D-b_{k})\\
&\le\frac{16\pi D}{k^{2}}.
\end{align*}
This completes the proof of the first inequality. Similarly, note that

$$
M_{\infty}\setminus W_{\infty}=\{(s, \theta) \mid s\in [a_{\infty}, a_{k})\cup (b_{k}, b_{\infty}], \ \ \theta\in S^{2}\} \subset M_{\infty}.
$$
Thus, similarly we have
\begin{align*}
\vol(M_{\infty}\setminus W_{\infty})
&=\int^{a_{k}}_{a_{\infty}}\int_{\mathbb{S}^{2}}f^{2}_{\infty}(s) \ d{\rm vol}_{g_{\mathbb{S}^{2}}}ds+\int^{b_{\infty}}_{b_{k}}\int_{\mathbb{S}^{2}}f^{2}_{\infty}(s) \ d{\rm vol}_{g_{S^{2}}}ds\\
&<4\pi \frac{1}{k^{2}}(a_{k}-a_{\infty})+4\pi \frac{1}{k^{2}}(b_{\infty}-b_{k})\\
&\le\frac{4\pi D}{k^{2}}.
\end{align*}
This completes the proof of the second inequality.

\end{proof}

\subsection{Uniformly Bounded Volumes and Areas}
 \begin{lem} \label{lem-vol-W}
For each fixed $k>0$ we have uniform upper bounds on volumes
of the diffeomorphic regions:
\be
\vol(W_j)\le \vol(M_{j})\leq 4\pi D^{3},
\ee
and
\be
 \vol(W_\infty)\le 4\pi D^{3}.
\ee
\end{lem}

\begin{proof}
The first half of the first inequality is clear, since $W_{j}$ is an embedded Riemannian submanifold of $M_{j}$. Since $f_{j}$ converges to $f_{\infty}$  on $[0, D]$ and
$$0\leq f_{j}(s)\leq D, \qquad \forall s\in [0, D],$$
it follows that
$$
0\le f_{\infty}(s)\le D, \qquad \forall s\in [0, D].
$$
Thus,
$$
\vol(M_{j})=\int^{D}_{0}\int_{\mathbb{S}^{2}}f^{2}_{j}(s) \ d{\rm vol}_{g_{\mathbb{S}^{2}}}ds\leq 4\pi\int^{D}_{0}D^{2} \ ds=4\pi D^{3}.
$$
Similarly,
$$
\vol(W_{\infty})=\int^{b_{k}}_{a_{k}}\int_{\mathbb{S}^{2}}f^{2}_{\infty}(s) \ d{\rm vol}_{g_{\mathbb{S}^{2}}}ds \le 4\pi \int^{D}_{0}D^{2} \ ds=4\pi D^{3}.
$$
\end{proof}

\begin{lem} \label{lem-area-W}
For each fixed $k>0$ we have uniform upper bounds on the areas of the boundaries,
\be
\area(\partial W_{j})\le 8\pi D^{2},
\ee
\be
\area(\partial W_{\infty})\le 8\pi D^{2}.
\ee
\end{lem}

\begin{proof}
For each fixed $k>0$,
$$
\partial W_{j}=(\{a_{k}\}\times \mathbb{S}^{2}, f^{2}_{j}(a_{k})g_{\mathbb{S}^{2}})\cup(\{b_{k}\}\times \mathbb{S}^{2}, f^{2}_{j}(b_{k})g_{\mathbb{S}^{2}}).
$$
Thus,
$$\area(\partial W_{j})=4\pi f^{2}_{j}(a_{k})+4\pi f^{2}_{j}(b_{k})\le 4\pi D^{2}+4\pi D^{2}=8\pi D^{2}.$$

Moreover,
$$
\partial W_{\infty}=(\{a_{k}\}\times \mathbb{S}^{2}, f^{2}_{\infty}(a_{k})g_{\mathbb{S}^{2}})\cup(\{b_{k}\}\times \mathbb{S}^{2}, f^{2}_{\infty}(b_{k})g_{\mathbb{S}^{2}})/
$$
Thus,
$$\area(\partial W_{\infty})=4\pi f^{2}_{\infty}(a_{k})+4\pi f^{2}_{\infty}(b_{k})\le 4\pi D^{2}+4\pi D^{2}=8\pi D^{2}.$$
\end{proof}

\subsection{Diameter and Distance bounds}

In this subsection we prove the SWIF convergence to $M_\infty$. We begin by estimating $h,\bar{h}, \lambda, h, a$ appearing in Theorem \ref{thm-subdiffeo}. Following Theorem \ref{thm-subdiffeo}, define $D_0$ by
\be
D_0= \max\{D, \diam(M_\infty)\}.
\ee
Then defining $a$ as \be
a>\frac{\arccos(1+\varepsilon)^{-1} }{\pi}D_0,
\ee
we may take $a$ arbitrarily small depending only on $\varepsilon$.

\begin{lem} \label{lem-h}
For each fixed large $k>0$, there exists $j_{0}(k)$ such that for all $j>j_{0}(k)$,
\be
\lambda=\sup_{x,y \in W}
\left|d_{M_j}(\psi_j(x),\psi_j(y))-d_{M_\infty}(\psi_\infty(x),\psi_\infty(y))\right| \le \frac{D+\diam(M_{\infty})+8\pi}{k-1}.
\ee
Thus $h$ and $\bar{h}$ can be made arbitrarily small.
\end{lem}

\begin{proof}
Because $f_{j}$ uniformly converges to $f_{\infty}$ on $[0, D]$, passing to a subsequence if necessarily, for any fixed large $k$, there exists $j_{0}(k)$ such that for all $j>j_{0}(k)$ and $s\in [0, D]$ we have
$$
\left|f_{j}(s)-f_{\infty}(s)\right|<\frac{1}{k(k+1)}.
$$
Thus, for all $s\in \overline{[0, D]\setminus I_{k}}=[0, a_{k}]\cup[b_{k}, D]$, and all $j>j_{0}(k)$,
$$
f_{j}(s)<f_{\infty}(s)+\frac{1}{k(k+1)}\leq\frac{1}{k}+\frac{1}{k(k+1)}<\frac{2}{k}.
$$
Moreover, because $\ds{f_{\infty}(s)\geq\frac{1}{k}}$ for all $s\in I_{k}$, we have that for all $j>j_{0}(k)$ and $s\in I_{k}$,
$$
f_{j}(s)<f_{\infty}(s)+\frac{1}{k(k+1)}\leq f_{\infty}(s)+f_{\infty}(s)\frac{1}{k+1}\leq\left(1+\frac{1}{k}\right)f_{\infty}(s),
$$
on the other hand,
$$f_{j}(s)>f_{\infty}(s)-\frac{1}{k(k+1)}\geq f_{\infty}(s)-f_{\infty}(s)\frac{1}{k+1}=\frac{k}{k+1}f_{\infty}(s)=\frac{1}{1+\frac{1}{k}}f_{\infty}(s).
$$

Thus, on $W=I_{k}\times \mathbb{S}^{2}$, for all $j>j_{0}(k)$, we have
$$
\left(\frac{1}{1+\frac{1}{k}}\right)^{2}g_{\infty}\leq g_{j} \leq \left(1+\frac{1}{k}\right)^{2}g_{\infty}.
$$
Then for any $x, y\in W$, and any piecewise smooth path $\gamma(t)$ lying in $W$ connecting $x$ and $y$, we have
$$
\left(1-\frac{\frac{1}{k}}{1+\frac{1}{k}}\right)\sqrt{g_{\infty}(\gamma^{\prime}(t), \gamma^{\prime}(t))}\leq\sqrt{g_{j}(\gamma^{\prime}(t), \gamma^{\prime}(t))}\leq \left(1+\frac{1}{k}\right)\sqrt{g_{\infty}(\gamma^{\prime}(t), \gamma^{\prime}(t))}.
$$
Therefore
$$
-\frac{1}{k}\sqrt{g_{\infty}(\gamma^{\prime}(t), \gamma^{\prime}(t))}\leq \sqrt{g_{j}(\gamma^{\prime}(t), \gamma^{\prime}(t))}-\sqrt{g_{\infty}(\gamma^{\prime}(t), \gamma^{\prime}(t))}\leq \frac{1}{k}\sqrt{g_{\infty}(\gamma^{\prime}(t), \gamma^{\prime}(t))},
$$
that is,
\be\label{TNC}
|\sqrt{g_{j}(\gamma^{\prime}(t), \gamma^{\prime}(t))}-\sqrt{g_{\infty}(\gamma^{\prime}(t), \gamma^{\prime}(t))}|\leq \frac{1}{k}\sqrt{g_{\infty}(\gamma^{\prime}(t), \gamma^{\prime}(t))}.
\ee

Now let $L_{j}(\gamma)$ and $L_{\infty}(\gamma)$ denote the length of the path $\gamma$ with respect to Riemannian metrics on $(M_{j}, g_{j})$ and $(M_{\infty}, g_{\infty})$ respectively. Then by $(\ref{TNC})$, we have
\be\label{Length-difference}
\left|L_{j}(\gamma)-L_{\infty}(\gamma)\right|\leq\frac{1}{k}L_{\infty}(\gamma),
\ee
for any path $\gamma$ in $W$.

Let $\gamma_{1}: [0, 1]\rightarrow M_{j}=[0, D_j]\times \mathbb{S}^{2}$ be the path that realizes the distance of $\psi_{j}(x)$ and $\psi_{j}(y)$ in $(M_{j}, g_{j})$ so that $L_{j}(\gamma_{1})=d_{M_{j}}(\psi_{j}(x), \psi_{j}(y))$. Let $\gamma_{2}$ be the path in $M_{\infty}=I_{\infty}\times \mathbb{S}^{2}$ that realizes the distance of $\psi_{\infty}(x)$ and $\psi_{\infty}(y)$ in $(M_{\infty}, g_{\infty})$ so that $L_{\infty}(\gamma_{2})=d_{M_{\infty}}(\psi_{\infty}(x), \psi_{\infty}(y))$.

Note $\gamma_{1}$ may not entirely lie in $W$ even though its endpoints are in $W$. The boundary $\partial W=\{a_{k}\}\times\mathbb{S}^{2}\cup\{b_k\}\times\mathbb{S}^{2}$ of $W$ has two connected components. If $\gamma_{1}$ does not entirely lie in $W$, there are two possibilities: either the first and the last intersecting points of $\gamma_{1}$ and $\partial W$ are in the same component, or the first and the last intersecting points of $\gamma_{1}$ and $\partial W$ are in different components.

In the first case, without loss of generality, we may assume that the first and the last intersecting points of $\gamma_{1}$ and $\partial W$ are $p=\gamma_{1}(t_{1})$ and $q=\gamma_{1}(t_{2})$ and both in $\{a_{k}\}\times\mathbb{S}^{2}$. Here $t_{1}<t_{2}$. Now we replace $\gamma_{1}\vert_{[t_{1}, t_{2}]}$ by the shortest geodesic $\delta$ connecting $p$ and $q$ in $\{a_{k}\}\times\mathbb{S}^{2}$, whose length is less than $\frac{2\pi}{k}$, since the diameter of $\{a_{k}\}\times\mathbb{S}^{2}$ is less than $\frac{2\pi}{k}$. Then we obtain a piece-wise smooth curve $\overline{\gamma}_{1}=\gamma_{1}\vert_{[0, t_{1}]}\cup\delta\cup\gamma_{1}\vert_{[t_{2}, 1]}$, which is in $W$.

In the second case, without loss of generality, we say that the first intersecting point of $\gamma_{1}$ and $\partial W$ is $p=\gamma_{1}(t_{1})\in\{a_{k}\}\times\mathbb{S}^{2}$ and the last intersecting point of $\gamma_{1}$ and $\partial W$ is $q=\gamma_{1}(t_{4})\in\{b_{k}\}\times\mathbb{S}^{2}$. Let $p^{\prime}=\gamma_{1}(t_{2})$ be the last intersecting point of $\gamma_{1}$ and $\{a_{k}\}\times\mathbb{S}^{2}$, and $q^{\prime}=\gamma_{1}(t_{3})$ be the first intersecting point of $\gamma_{1}$ and $\{b_{k}\}\times\mathbb{S}^{2}$ after $p^{\prime}=\gamma_{1}(t_{2})$. Note $0\leq t_{1}\leq t_{2}<t_{3}\leq t_{4}\leq1$. Now we replace $\gamma_{1}\vert_{[t_{1}, t_{2}]}$ by the shortest geodesic $\delta$ connecting $p$ and $p^{\prime}$ in $\{a_{k}\}\times\mathbb{S}^{2}$, whose length is less than $\frac{2\pi}{k}$. Similarly, we replace $\gamma_{1}\vert_{[t_{3}, t_{4}]}$ by the shortest geodesic $\delta^{\prime}$ connecting $q$ and $q^{\prime}$ in $\{b_{k}\}\times\mathbb{S}^{2}$, whose length is also less than $\frac{2\pi}{k}$. Then we obtain a piece-wise smooth curve $\overline{\gamma}_{1}=\gamma_{1}\vert_{[0, t_{1}]}\cup\delta\cup\gamma_{1}\vert_{[t_{2}, t_{3}]}\cup\delta^{\prime}\cup\gamma_{1}\vert_{[t_{4}, 1]}$, which is in $W$.

Similarly, if $\gamma_{2}$ is not entirely in $W$, then we do the same thing as above for $\gamma_{2}$ to obtain $\overline{\gamma}_{2}$, which is in $W$.  Clearly, $\overline{\gamma}_{1}, \overline{\gamma}_{2}\subset W$. Moreover, by the construction of $\overline{\gamma}_{1}$ and $\overline{\gamma}_{2}$, we can easily obtain
\be\label{inequalities-for-length}
\begin{aligned}
L_{j}(\gamma_{1})\leq L_{j}(\overline{\gamma}_{1}),
& & L_{\infty}(\gamma_{2})\leq L_{\infty}(\overline{\gamma}_{2}),\\
L_{j}(\overline{\gamma}_{1})-\frac{4\pi}{k}\leq L_{j}(\gamma_{1}),
& & L_{\infty}(\overline{\gamma}_{2})-\frac{2\pi}{k}\leq L_{\infty}(\gamma_{2}).
\end{aligned}
\ee
Here, in the last inequality we used the fact that \be \nonumber f_{\infty}(a_{k})=f_{\infty}(b_{k})=\frac{1}{k}. \ee

If $\gamma_{1}$ (or $\gamma_{2}$) already entirely lies in $W$, then we keep it and simply say $\overline{\gamma}_{1}=\gamma_{1}$ (or $\overline{\gamma}_{2}=\gamma_{2}$) so the inequalities in $(\ref{inequalities-for-length})$ still hold.

Finally, by using inequalities in $(\ref{Length-difference})$ and $(\ref{inequalities-for-length})$, we have
\begin{align*}
&\quad d_{M_{j}}(\psi_{j}(x), \psi_{j}(y))-d_{M_{\infty}}(\psi_{\infty}(x), \psi_{\infty}(y))\\
&= L_{j}(\gamma_{1})-L_{\infty}(\gamma_{2})\\
&\leq L_{j}(\overline{\gamma}_{1})-L_{\infty}(\overline{\gamma}_{2})+\frac{2\pi}{k}\\
&\leq \frac{1}{k}L_{\infty}(\overline{\gamma}_{2})+\frac{2\pi}{k}\\
&\leq \frac{1}{k}L_{\infty}(\gamma_{2})+\frac{2\pi}{k^{2}}+\frac{\pi}{k}\\
&\leq \frac{\diam(M_{\infty})+4\pi}{k}
\leq \frac{D+\diam(M_{\infty})+8\pi}{k-1},
\end{align*}
on the other hand,
\begin{align*}
&\quad d_{M_{j}}(\psi_{j}(x), \psi_{j}(y))-d_{M_{\infty}}(\psi_{\infty}(x), \psi_{\infty}(y))\\
&= L_{j}(\gamma_{1})-L_{\infty}(\gamma_{2})\\
&\geq L_{j}(\overline{\gamma}_{1})-\frac{4\pi}{k}-L_{\infty}(\overline{\gamma}_{2})\\
&\geq -\frac{1}{k}L_{\infty}(\overline{\gamma}_{1})-\frac{4\pi}{k}\\
&\geq -\frac{1}{k-1}L_{j}(\overline{\gamma}_{1})-\frac{4\pi}{k}\\
&\geq -\frac{1}{k-1}L_{j}(\gamma_{1})-\frac{4\pi}{k(k-1)}-\frac{4\pi}{k}\\
&\geq -\frac{\diam(M_{j})+8\pi}{k-1}
\geq -\frac{D+\diam(M_{\infty})+8\pi}{k-1}.
\end{align*}
This completes the proof.
\end{proof}

Finally we can prove the following theorem applying these lemmas and carefully
balancing the choice of $k$ and taking $j$ large enough.

\begin{thm}\label{swif} Under the assumptions in Theorem $\ref{main-result}$, there exists a subsequence of $\{M_{j}\}$ that converges to
$M_\infty$ in SWIF sense.
\end{thm}

\begin{proof}
By Lemma $\ref{lem-vol-W}$ and Lemma $\ref{lem-area-W}$, for all $j$,
$$
\vol(W_{j})\leq 4\pi D^3,  \qquad \area(\partial W_{j})\leq 8\pi D^{2},
$$
and
$$
\vol(W_{\infty})\leq 4\pi D^{3}, \qquad \area(\partial W_{\infty})\leq 8\pi D^{2}.
$$
Set $D_{0}=\max\{D, \diam(M_{\infty})\}$.
Choose a sufficiently large integer $k$ such that $\ds{\arccos\left(\frac{1}{1+\frac{1}{k}}\right)<\frac{2}{\sqrt{k}}}$, which implies $\ds{\arccos\left(\frac{1}{1+\frac{1}{k+i}}\right)<\frac{2}{\sqrt{k+i}}}$, for all positive integers $i$. Now take
$$
a=\frac{2 D}{\pi\sqrt{k+1}}.
$$
There exists a subsequence $\{M^{(1)}_{1}, M^{(1)}_{2}, M^{(1)}_{3}, \cdots\}$ of $\{M_{j}\}$ such that
$$
\vol(M^{(1)}_{j}\setminus (W^{(1)}_{j})^{k+1})\leq\frac{4\pi D}{(k+1)^{2}}, \qquad \vol(M_{\infty}\setminus W^{k+1}_{\infty})\leq \frac{4\pi D}{(k+1)^{2}},
$$
and
$$
\lambda\leq\frac{\diam(M_{\infty})}{k+1}.
$$
Then
$$
h\leq\sqrt{2\lambda D_{0}}\leq\frac{2D_{0}}{\sqrt{k+1}},
$$
and
$$
\bar{h}=\max\Bigg\{h, D_{0}\sqrt{\frac{1}{(k+1)^{2}}+\frac{2}{k+1}}\Bigg\}\leq\frac{2D_{0}}{\sqrt{k+1}}.
$$
By Theorem $\ref{thm-subdiffeo}$,
\begin{eqnarray*}
d_{\mathcal{F}}(M^{(1)}_j, M_\infty)
&\le&
\left(2\bar{h} + a\right) \Big[
\vol((W^{1}_{j})^{k+1})+\vol(W^{k+1}_\infty)+\area(\partial (W^{(1)}_{j})^{k+1})\\
& &+\area(\partial W^{k+1}_{\infty})\Big]
+\vol(M^{(1)}_j\setminus (W^{(1)}_j)^{k+1})+\vol(M_\infty\setminus W^{k+1}_\infty)\\
&\le& \left(\frac{4D_{0}}{\sqrt{k+1}}+\frac{2D}{\pi\sqrt{k+1}}\right)(8\pi D^3+16\pi D^{2})+\frac{8\pi D}{(k+1)^{2}}\\
&\le& \left(\frac{1}{\sqrt{k+1}}\right)\left[\left(4D_{0}+\frac{2D}{\sqrt{\pi}}\right)(8\pi D^3+16\pi D^{2})+8\pi D\right],
\end {eqnarray*}
for all $j$.

Similarly, we can take a subsequence
$$\{M^{(2)}_{1}, M^{(2)}_{2}, M^{(2)}_{3}, \cdots\}\subset\{M^{(1)}_{1}, M^{(1)}_{2}, M^{(1)}_{3}, \cdots\}$$
such that
$$
d_{\mathcal{F}}(M^{(2)}_j, M_\infty)\le \left(\frac{1}{\sqrt{k+2}}\right)\left[\left(4D_{0}+\frac{2D}{\sqrt{\pi}}\right)(8\pi D^3+16\pi D^{2})+8\pi D\right],
$$
for all $j$.

Repeating this process, we have for each positive integer $i$, there are subsequence $\{M^{(i+1)}_{1}, M^{(i+1)}_{2}, M^{(i+1)}_{3}, \cdots\}\subset\{M^{(i)}_{1}, M^{(i)}_{2}, M^{(i)}_{3}, \cdots\}$ such that
$$
d_{\mathcal{F}}(M^{(i+1)}_j, M_\infty)\le \left(\frac{1}{\sqrt{k+i+1}}\right)\left[\left(4D_{0}+\frac{2D}{\sqrt{\pi}}\right)(8\pi D^3+16\pi D^{2})+8\pi D\right],
$$
for all $j$.

Finally, we take the subsequence $\{M^{(1)}_{1}, M^{(2)}_{2}, M^{(3)}_{3}, \cdots\}$. For any $\varepsilon>0$, there exists a positive integer $i_{0}$ such that for all $i>i_{0}$,
$$
\left(\frac{1}{\sqrt{k+i}}\right)\left[\left(4D_{0}+\frac{2D}{\sqrt{\pi}}\right)(8\pi D^3+16\pi D^{2})+8\pi D\right]<\varepsilon.
$$
Thus, $d_{\mathcal{F}}(M^{(i)}_i, M_\infty)<\varepsilon$ for all $i>i_{0}$. This completes the proof.
\end{proof}

\end{document}